\documentclass{amsart}
\usepackage{amsfonts, amsbsy, amsmath, amssymb}

\usepackage{blindtext}
\usepackage{scrextend}
\addtokomafont{labelinglabel}{\sffamily}

\usepackage{mathdots}

\input prepictex.tex
\input pictex.tex
\input postpictex.tex

\newtheorem{thm}{Theorem}[section]
\newtheorem{lem}[thm]{Lemma}

\newtheorem{prop}[thm]{Proposition}
\newtheorem{exmp}[thm]{Example}

\newtheorem{algo}[thm]{Algorithm}
\newtheorem{rmk}[thm]{Remark}

\newtheorem{thm-con}[thm]{Theorem-Conjecture}
\numberwithin{equation}{section}

\theoremstyle{definition}

\allowdisplaybreaks

\newcommand{\f}{\Bbb F}
\begin{document}

\title{On Gauss Periods}

\author[Xiang-dong Hou]{Xiang-dong Hou}
\address{Department of Mathematics and Statistics,
University of South Florida, Tampa, FL 33620}
\email{xhou@usf.edu}
%\thanks{* Research partially supported by .}


\subjclass[2000]{}

\begin{abstract}
Let $q$ be a prime power, and let $r=nk+1$ be a prime such that $r\nmid q$, where $n$ and $k$ are positive integers. Under a simple condition on $q$, $r$ and $k$, a Gauss period of type $(n,k)$ is a normal element of $\f_{q^n}$ over $\f_q$; the complexity of the resulting normal basis of $\f_{q^n}$ over $\f_q$ is denoted by $C(n,k;q)$. Recent works determined $C(n,k;q)$ for $k\le 7$ and all qualified $n$ and $q$. In this paper, we show that for any given $k>0$, $C(n,k;q)$ is given by an explicit formula except for finitely many primes $r=nk+1$ and the exceptional primes are easily determined. Moreover, we describe an algorithm that allows one to compute $C(n,k;q)$ for the exceptional primes $r=nk+1$. The numerical results of the paper cover $C(n,k;q)$ for $k\le 20$ and all qualified $n$ and $q$.
\end{abstract}

\maketitle

%%%%%%%%%%%%%%%%%%%%%%%%%%%%%%%%%%%%%%%%%%%
%  section 1
%%%%%%%%%%%%%%%%%%%%%%%%%%%%%%%%%%%%%%%%%%%
\section{Introduction and Preliminaries}

We briefly recall the existing results on Gauss periods, some of which are quite recent.

Let $q$ be a prime power, and let $r=nk+1$ be a prime such that $r\nmid q$, where $n$ and $k$ are positive integers. Since $k\mid r-1$, there is a unique cyclic subgroup $K$ of $\Bbb Z_r^*$ with $|K|=k$. Since $q^{nk}-1=q^r-1\equiv 0\pmod r$, there exists $\gamma\in\f_{q^{nk}}^*$ such that $o(\gamma)=r$. Let 
\begin{equation}\label{1.1}
\alpha_i=\sum_{a\in K}\gamma^{q^ia},\quad 0\le i\le n-1,
\end{equation}
which are called {\em Gauss periods} of type $(n,k)$ over $\f_q$. Note that $\alpha_i=\alpha_0^{q^i}$. Since $(q^n)^k\equiv 1\pmod r$, we have $q^n\in K$. Therefore
\[
\alpha_i^{q^n}=\sum_{a\in K}\gamma^{q^iq^na}=\sum_{a\in K}\gamma^{q^ia}=\alpha_i,
\]
and hence $\alpha_i\in\f_{q^n}$.

A necessary condition for $\alpha_0,\dots,\alpha_{n-1}$ to form a normal basis of $\f_{q^n}$ over $\f_q$ is that these elements be distinct. This condition requires $q^iK$, $0\le i\le n-1$, to be distinct cosets of $K$ in $\Bbb Z_r^*$, i.e., $\langle q,K\rangle=\Bbb Z_r^*$, which happens if and only if $\text{gcd}(nk/e,n)=1$, where $e=e(q,r)$ is the order of $q$ in $\Bbb Z_r^*$. 

On the other hand, if $\langle q,K\rangle=\Bbb Z_r^*$, then $\alpha_0,\dots,\alpha_{n-1}$ indeed form a normal basis of $\f_{q^n}$ over $\f_q$; the reason is the following \cite{Gao-thesis-1993, Wassermann-BMS-1990}. Assume that $\sum_{i=0}^{n-1}c_i\alpha_i=0$, where $c_i\in\f_q$, i.e.,
\[
\sum_{i=0}^{n-1}\sum_{a\in q^iK}c_i\gamma^a=0.
\]
Treat each $a\in\Bbb Z_r^*$ as an interger in $\{1,\dots,r-1\}$, and let
\[
f(X)=\sum_{i=0}^{n-1}\sum_{a\in q^iK}c_iX^a\in\f_q[X].
\]
For each $m\in\Bbb Z_r^*$, we have $m=q^ja_0$ for some $0\le j\le n-1$ and $a_0\in K$. Thus
\[
f(\gamma^m)=\sum_{i=0}^{n-1}\sum_{a\in q^iK}c_i\gamma^{q^ja_0a}=\Bigl(\sum_{i=0}^{n-1}\sum_{a\in q^iK}c_i\gamma^a\Bigr)^{q^j}=f(\gamma)^{q^j}=0.
\]
Also note that $f(0)=0$. Therefore $f(x)=0$ for all $x\in\{0\}\cup\langle\gamma\rangle$. Since $\deg f\le r-1$, we must have $f=0$, i.e., $c_i=0$ for all $0\le i\le n-1$.

From now one, we assume that $\langle q,K\rangle=\Bbb Z_r^*$ and hence $\alpha_0,\dots,\alpha_{n-1}$ form a normal basis of $\f_{q^n}$ over $\f_q$. To avoid triviality, we also assume that $n>1$, so $nk$ is even. For $0\le i\le n-1$,
\begin{align}\label{1.2}
\alpha_0\alpha_i\,&=\sum_{a,b\in K}\gamma^{a+q^ib}=\sum_{a,b\in K}\gamma^{a(1+q^ib)}\cr
&=k|\{-q^{-i}\}\cap K|+\sum_{j=0}^{n-1}\alpha_j|(1+q^iK)\cap q^jK|.
\end{align}
In the above, 
\begin{equation}\label{1.3}
t_{ij}=|(1+q^iK)\cap q^jK|,\quad 0\le i,j\le n-1,
\end{equation}
are called the {\em cyclotomic numbers}. Note that
\[
q^{-i}\in -K\Leftrightarrow
\begin{cases}
q^{-i}\in K\Leftrightarrow i=0&\text{if $k$ is even},\cr
q^{-2i}\in K\ \text{but}\ q^{-i}\notin K\Leftrightarrow i=n/2&\text{if $k$ is odd}.
\end{cases}
\]
Hence $|\{-q^{-i}\}\cap K|=\delta_i$, where
\begin{equation}\label{1.4.0}
\delta_i=
\begin{cases}
1&\text{if $k$ is even and $i=0$, or $k$ is odd and $i=n/2$},\cr
0&\text{otherwise}.
\end{cases}
\end{equation}
Now \eqref{1.2} becomes
\begin{equation}\label{1.4}
\alpha_0\alpha_i=k\delta_i+\sum_{j=0}^{n-1}t_{ij}\alpha_j=\sum_{j=0}^{n-1}(t_{ij}-k\delta_i)\alpha_j,\quad 0\le i\le n-1.
\end{equation}
(The last step in the above follows from the fact that $\sum_{j=0}^{n-1}\alpha_j=\sum_{a\in\Bbb Z_r^*}\gamma^a=-1$.) Let $p=\text{char}\,\f_q$ and let
\begin{equation}\label{1.5}
C(n,k;q)=\bigl|\{(i,j): 0\le i,j\le n-1,\ t_{ij}-k\delta_i\not\equiv 0\pmod p\}\bigr|,
\end{equation}
which is the number of nonero entries of the matrix $(t_{ij}-k\delta_i)_{0\le i,j\le n-1}$ over $\f_p$ and is called the {\em complexity} of the normal basis $\alpha_0,\dots,\alpha_{n-1}$ of $\f_{q^n}$ over $\f_q$. Normal bases with low complexity are sought as they provide effective algorithms for multiplication and exponentiation in finite fields \cite{Gao-Gathen-Panario-Shoup-JSC-2000}.

Equation \eqref{1.5} involves two primes: $p=\text{char}\,\f_q$ and $r=nk+1$. The prime $p$ only plays a superficial role since the cyclotomic numbers $t_{ij}$ depend only on $r$ and $k$. 

The values of $C(n,k;q)$ were first determined in \cite{Mullin-Onyszchuk-Vanstone-Wilson-DAM-1988} for $k=1$ and for $k=2$ with an even $q$. (The case $k=2$ with an arbitrary $q$ is covered by a recent result of \cite{Liao-Hu-AMS-2014}.) A more careful investigation of the cyclotomic numbers $t_{ij}$ in \cite{Christopoulou-Garefalakis-Panario-Thomson-DCC-2012} produced explicit formulas for $C(n,k;q)$ with $3\le k\le 6$ and $n\ge 3$. More recently, the following general result was proved in \cite{Liao-Hu-AMS-2014}: 

\begin{thm}[{\cite[Theorem~1.7]{Liao-Hu-AMS-2014}}]\label{T1.1}
Assume that $t_{ij}\le 2$ for all $0\le i,j\le n-1$.
\begin{itemize}
\item[(i)] If $k$ is even,
\[
C(n,k;q)=
\begin{cases}
nk-k^2+3k-3&\text{if $p=2$},\vspace{2mm}\cr
\displaystyle nk-\frac 12 k^2+\frac 32k-2&\text{if $p>2$ and $k\equiv 0\pmod p$},\vspace{2mm}\cr
\displaystyle n(k+1)-\frac 12k^2+k-3&\text{if $p>2$ and $k\equiv 1\pmod p$},\vspace{2mm}\cr
\displaystyle n(k+1)-\frac 12k^2+\frac 12k-1&\text{if $p>2$ and $k\equiv 2\pmod p$},\vspace{2mm}\cr
\displaystyle n(k+1)-\frac 12k^2+k-2&\text{if $p>2$ and $k\not\equiv 0,1,2\pmod p$}.
\end{cases}
\]
\item[(ii)] If $k$ is odd,
\[
C(n,k;q)=
\begin{cases}
n(k+1)-k^2+k-1&\text{if $p=2$},\vspace{2mm}\cr
\displaystyle nk-\frac 12 k^2+\frac 32k-2&\text{if $p>2$ and $k\equiv 0\pmod p$},\vspace{2mm}\cr
\displaystyle n(k+1)-\frac 12k^2-\frac 12 k&\text{if $p>2$ and $k\equiv 1\pmod p$},\vspace{2mm}\cr
\displaystyle n(k+1)-\frac 12k^2+\frac 12 k-1&\text{if $p>2$ and $k\not\equiv 0,1\pmod p$}.
\end{cases}
\]
\end{itemize}
\end{thm}

It is claimed in \cite{Liao-Hu-AMS-2014} that for $k=7$, the only exceptional case not covered by Theorem~\ref{T1.1} is $n=4$ ($r=29$). We will see that for $k=7$, \cite{Liao-Hu-AMS-2014} missed another exceptional case ($n=6$, $r=43$); for $k=6$, \cite{Christopoulou-Garefalakis-Panario-Thomson-DCC-2012} also missed an exceptional case ($n=3$, $r=19$).

The main contribution of the present paper is the following: For each given $k\ge 1$, we show that except for finitely many primes $r$, which we shall call {\em exceptional primes}, the condition $t_{ij}\le 2$ ($0\le i,j\le n-1$) is satisfied and hence $C(n,k;q)$ is given by Theorem~\ref{T1.1}. The exceptional primes are easily determined using resultants in characteristic $0$. Moreover, for each exceptional prime $r$, we describe an algorithm for computing the value distribution of the cyclotomic numbers, which then gives $C(n,k;q)$. The combined message is that for any given $k$ (not too big), $C(n,k;q)$ can be determined for all qualified $n$ and $q$. We demonstrate the computational results of $C(n,k;q)$ for $k\le 20$.

Section~2 contains a review and a further discussion of the cyclotomic numbers. In Section~3 we prove that for each given $k$, there are only finitely many exceptional primes $r$, and we explain where these exceptional primes come from and how to find them. In Section~4 we describe an algorithm that allows one to compute $C(n,k;q)$ for a given $k$, an exceptional prime $r=nk+1$, and all qualified $q$. The computational results of $C(n,k;q)$ for $k\le 20$ are included. We also give a theoretic explanation for the computational results with $n=2$.

%%%%%%%%%%%%%%%%%%%%%%%%%%%%%%%%%%%%%%%%%%%
%  section 2
%%%%%%%%%%%%%%%%%%%%%%%%%%%%%%%%%%%%%%%%%%%
\section{Cyclotomic Numbers}

Cyclotomic numbers have been studied in terms of Jacobi sums by many authors; see for example \cite{Acharya-Katre-AA-1995, Baumert-Fredricksen-1967, Dickson-AJM-1935, Katre-Rajwade-AA-1985, Muskat-AA-1966, Shirolkar-Katre-AA-2011, Whiteman-AA-1960}. The focus of the present paper is a little different. 
Most (but not all) of the results gathered in this section have appeared, explicitly or implicitly, in \cite{Christopoulou-Garefalakis-Panario-Thomson-DCC-2012, Liao-Hu-AMS-2014}.

Recall that $n>1$, $r=nk+1$ is a prime, $K$ is the unique subgroup of $\Bbb Z_r^*$ of order $k$, and $\langle q,K\rangle=\Bbb Z_r^*$. Write $K_i=q^iK$, $0\le i\le n-1$, hence 
\begin{equation}\label{2.1}
t_{ij}=|(1+K_i)\cap K_j|,\quad 0\le i,j\le n-1.
\end{equation}

\begin{lem}\label{L2.1}
$t_{ij}>0$ if and only if there exists $x\in\Bbb Z_r\setminus\{0,-1\}$ such that $K_i=xK$ and $K_j=(1+x)K$.
\end{lem}

\begin{proof} ($\Leftarrow$) We have $1+x\in(1+K_i)\cap K_j$, hence $t_{ij}>0$.

\medskip
($\Rightarrow$) Let $y\in(1+K_i)\cap K_j$ and $x=y-1$. Then $x\in K_i$ and $1+x\in K_j$.
\end{proof}

For $x,y\in\Bbb Z_r\setminus\{0.-1\}$, define $x\sim y$ if and only if $xK=yK$ and $(1+x)K=(1+y)K$, i.e., $x/y\in K$ and $(1+x)/(1+y)\in K$. Let $[x]$ denote the $\sim$ equivalence class of $x$. For conveninence, we also define $[0]=\{0\}$ and $[-1]=\{-1\}$.

\begin{lem}\label{L2.2}
For each $x\in\Bbb Z_r$,
\begin{equation}\label{2.2}
(1+xK)\cap(1+x)K=1+[x].
\end{equation}
\end{lem}

\begin{proof} Let $y\in\Bbb Z_r$. If $x\ne 0,-1$,
\[
\begin{array}{rcl}
1+y\in(1+xK)\cap (1+x)K&\Leftrightarrow & y\in xK\ \text{and}\ 1+y\in(1+x)K\cr
&\Leftrightarrow & y\sim x.
\end{array}
\]
If $x=0$,
\[
1+y\in(1+0K)\cap (1+0)K\ \Leftrightarrow\ y=0.
\]
If $x=-1$,
\[
1+y\in(1-K)\cap 0K\ \Leftrightarrow\ y=-1.
\]
\end{proof}

Let $K=\langle\omega\rangle$ and let $\mathcal S_k=\{(u,v)\in\Bbb Z_k^2:u\ne 0,\ v\ne 0,\ u\ne v\}$. Define 
\begin{equation}\label{2.3}
\begin{array}{cccc}
S:&\mathcal S_k&\longrightarrow & \Bbb Z_r^*\setminus\{-1\}\vspace{2mm}\cr
&(u,v)&\longmapsto &\displaystyle -\frac{1-\omega^v}{\omega^u-\omega^v}.
\end{array}
\end{equation}

\begin{lem}\label{L2.3}
Let $x,y\in\Bbb Z_r^*\setminus\{-1\}$. Then $x\sim y$ but $x\ne y$ if and only if there exists $(u,v)\in S^{-1}(x)$ such that $y=S(-u,-v)$.
\end{lem}

\begin{proof}
($\Rightarrow$) Since $x\sim y$, we have 
\begin{equation}\label{2.4}
\begin{cases}
\displaystyle\frac yx=\omega^u,\vspace{2mm}\cr
\displaystyle\frac{1+y}{1+x}=\omega^v,
\end{cases}
\end{equation}
for some $u,v\in\Bbb Z_k$. Since $x,y\in\{0,-1\}$ and $x\ne y$, we have $u\ne 0$, $v\ne 0$ and $u\ne v$. Solving \eqref{2.4} gives $x=S(u,v)$ and $y=S(-u,-v)$.

\medskip
($\Rightarrow$) \eqref{2.4} is satisfied for $x=S(u,v)$ and $y=S(-u,-v)$. Since $u\ne 0$, $y\ne x$.
\end{proof}

\begin{prop}\label{P2.4}
Let $x\in\Bbb Z_r^*\setminus\{-1\}$. Then 
\begin{equation}\label{2.5}
|[x]|=1+|S^{-1}(x)|.
\end{equation}
\end{prop}

\begin{proof}
By Lemma~\ref{L2.3},
\[
[x]=\{x\}\overset{\boldsymbol\cdot}\cup \{S(-u,-v):(u,v)\in S^{-1}(x)\},
\]
where $\overset{\boldsymbol\cdot}\cup$ means disjoint union. It remains to show that the mapping
\[
\begin{array}{ccc}
S^{-1}(x)&\longrightarrow &\Bbb Z_r^*\setminus\{-1\}\vspace{2mm}\cr
(u,v)&\longmapsto & S(-u,-v)
\end{array}
\]
is one-to-one. Let $(u_1,v_1),(u_2,v_2)\in S^{-1}(x)$ be such that $S(-u_1,-v_1)=S(-u_2,-v_2)$. Then
\[
\omega^{u_1}S(u_1,v_1)=S(-u_1,-v_1)=S(-u_2,-v_2)=\omega^{u_2}S(u_2,v_2).
\]
It follows that $u_1=u_2$. Then $S(u_1,v_1)=S(u_1,v_2)$, which gives $v_1=v_2$.
\end{proof}

\begin{rmk}\label{R2.5}\rm
It follows from Lemma~\ref{L2.1}, \eqref{2.2} and \eqref{2.5} that $t_{ij}\le 2$ for all $0\le i,j\le n-1$ if and only if the mapping $S:\mathcal S_k\to\Bbb Z_r^*\setminus\{-1\}$ is one-to-one.
\end{rmk}

For $0\le i\le n-1$ and $0\le \tau\le k$, define
\begin{align}\label{2.6}
a_i(\tau)\,&=|\{0\le j\le n-1:t_{ij}=\tau\}|,\\
\label{2.7}
a(\tau)\,&=\sum_{i=0}^{n-1}a_i(\tau).
\end{align}
Then
\begin{equation}\label{2.8}
\sum_\tau\tau a_i(\tau)=\sum_jt_{ij}=\sum_j|(1+K_i)\cap K_j|=|(1+K_i)\cap\Bbb Z_r^*|=k-\delta_i
\end{equation}
and
\begin{equation}\label{2.9}
\sum_\tau \tau a(\tau)=nk-1.
\end{equation}
Let 
\begin{equation}\label{2.9-1}
a_*(\tau)=
\begin{cases}
a_0(\tau)&\text{if $k$ is even},\cr
a_{n/2}(\tau)&\text{if $k$ is odd}.
\end{cases}
\end{equation}
It follows from \eqref{1.5} that
\begin{align}\label{2.9-2}
C(n,k;q)\,&=\sum_{\tau\not\equiv 0\,(\text{mod}\, p)}\bigl(a(\tau)-a_*(\tau)\bigr)+\sum_{\tau\not\equiv k\,(\text{mod}\, p)}a_*(\tau)\cr
&=n^2-\sum_{\tau\equiv 0\,(\text{mod}\, p)}\bigl(a(\tau)-a_*(\tau)\bigr)-\sum_{\tau\equiv k\,(\text{mod}\, p)}a_*(\tau),
\end{align}
where $p=\text{char}\,\f_q$.
Therefore, $C(n,k;q)$ is determined by $a(\tau)$ and $a_*(\tau)$, $0\le \tau\le k$.

\begin{prop}\label{P2.6}
\begin{itemize}
\item[(i)] Let $0\le i\le n-1$. We have
\begin{align}\label{2.10}
a_i(\tau)\,&=\frac 1\tau\bigl|\{x\in K_i\setminus\{-1\}: |S^{-1}(x)|=\tau-1\}\bigr|,\quad 2\le \tau\le k,\\
\label{2.11}
a_i(1)\,&=k-\delta_i-\sum_{\tau=2}^k\tau a_i(\tau),\\
\label{2.12}
a_i(0)\,&=n-\sum_{\tau=1}^ka_i(\tau).
\end{align}

\item[(ii)] We have
\begin{align}\label{2.13}
a(\tau)\,&=\frac 1\tau\bigl|\{x\in \Bbb Z_r^*\setminus\{-1\} : |S^{-1}(x)|=\tau-1\}\bigr|,\quad 2\le \tau\le k,\\
\label{2.14}
a(1)\,&=nk-1-\sum_{\tau=2}^k\tau a(\tau),\\
\label{2.15}
a(0)\,&=n^2-\sum_{\tau=1}^ka(\tau).
\end{align}
\end{itemize}
\end{prop}

\begin{proof} (i) We only have to prove \eqref{2.10}. (\eqref{2.11} follows from \eqref{2.8}, and \eqref{2.12} is obvious.) Let $2\le \tau\le k$ and let
\[
\mathcal X=\{x\in K_i\setminus\{-1\}: |S^{-1}(x)|=\tau-1\}.
\]
Define
\[
\begin{array}{cccl}
f:&\mathcal X& \longrightarrow &\{0\le j\le n-1:t_{ij}=\tau\}\vspace{2mm}\cr
&x&\longmapsto & j,\quad \text{where}\ 1+x\in K_j.
\end{array}
\]

\medskip
$1^\circ$ We claim that $f$ is well defined. Assume that $x\in\mathcal X$. Since $S^{-1}(x)\ne\emptyset$, $x\ne-1$, so $1+x\in K_j$ for some $0\le j\le n-1$. Note that
\begin{align*}
t_{ij}\,&=|(1+xK)\cap(1+x)K|\cr
&=|[x]|\kern 3cm \text{(by \eqref{2.2})}\cr
&=1+|S^{-1}(x)|\kern 1.7cm \text{(by \eqref{2.5})}\cr
&=\tau.
\end{align*}

\medskip
$2^\circ$ We claim that $f$ is onto. Assume that $t_{ij}=\tau$. Since $\tau\ge 2$, by Lemma~\ref{L2.1}, there exists $x\in\Bbb Z_r^*\setminus\{-1\}$ such that $K_i=xK$ and $K_j=(1+x)K$. Moreover, by \eqref{2.5} and \eqref{2.2},
\[
|S^{-1}(x)|=|[x]|-1=|((1+xK)\cap(1+x)K|-1=t_{ij}-1=\tau-1.
\]
Hence $x\in\mathcal X$ and $f(x)=j$.

\medskip
$3^\circ$ Let $x\in\mathcal X$ and $j=f(x)$. We claim that $f^{-1}(j)=[x]$; since $|[x]|=\tau$, this claim implies that $f$ is $\tau$-to-$1$. If $x'\in f^{-1}(j)$, then $x'\in K_i$ and $1+x'\in K_j=(1+x)K$, i.e., $x'\in[x]$. On the other hand, if $x'\in[x]$, then
\[
|S^{-1}(x')|=|[x']|-1=|[x]|-1=\tau-1.
\]
Hence $x'\in\mathcal X$. Since $1+x'\in(1+x)K=K_j$, we have $f(x')=j$, i.e., $x\in f^{-1}(j)$.

\medskip
$4^\circ$ Equation~\eqref{2.10} follows from $1^\circ$ -- $3^\circ$.

\medskip
(ii) The equations follow from (i).
\end{proof}

Assume that $t_{ij}\le 2$ for all $0\le i,j\le n-1$, i.e., $S:\mathcal S_k\to\Bbb Z_r^*\setminus\{-1\}$ is one-to-one. Under this assumption, we are able to determine $a(2)$ and $a_*(2)$ explicitly. By \eqref{2.13},
\begin{equation}\label{2.16}
a(2)=\frac 12|\mathcal S_k|=\frac 12(k-1)(k-2).
\end{equation}
We claim that $1\notin S(\mathcal S_k)$. If, to the contrary, $S(u,v)=1$ for some $(u,v)\in\mathcal S_k$, then $S(u,v)=S(u,v)^{-1}=S(-u,v-u)$, where $(-u,v-u)\in\mathcal S_k$. It follows that $(u,v)=(-u,v-u)$, which is a contradiction.

\medskip
(i) Assume that $k$ is even. Then $\{\omega^0,\omega^{k/2}\}=\{\pm1\}$ is disjoint from $S(\mathcal S_k)$. On the other hand, for all $u\in\Bbb Z_k\setminus\{0,k/2\}$, $\omega^u=S(-2u,-u)$, where $(-2u,-u)\in\mathcal S_k$. Hence
\begin{equation}\label{2.17}
a_0(2)=\frac12 (k-2).
\end{equation}

\medskip
(ii) Assume that $k$ is odd. We claim that $S(\mathcal S_k)\cap(-K)=\emptyset$. If, to the contrary, there exist $(u,v)\in\mathcal S_k$ and $l\in\Bbb Z_k$ such that $S(u,v)=-\omega^l$. It follows that 
\[
-\frac{1-\omega^v}{\omega^{2v-u}-\omega^v}=\omega^{l+u-v}=-\frac{1-\omega^{-(l+u-v)}}{\omega^{-2(l+u-v)}-\omega^{-(l+u-v)}}.
\]
In the above, $\omega^{l+u-v}\ne -1$ since $k$ is odd. Thus $2v-u\ne 0$ and $(2v-u,v)\in\mathcal S_k$. Moreover, $l+u-v\ne 0$ since otherwise, $S(2v-u,v)=1$, which is impossible. Therefore $(-2(l+u-v),-(l+u-v))\in\mathcal S_k$ and 
\[
S(2v-u,u)=S(-2(l+u-v),-(l+u-v)).
\]
Then $(2v-u,v)=(-2(l+u-v),-(l+u-v))$, which forces $u=0$, a contradiction. Therefore the claim is proved. We conclude that 
\begin{equation}\label{2.19}
a_{n/2}(2)=0.
\end{equation}

\begin{rmk}\label{R2.7}\rm
In the above, one can further determine $a(\tau)$ and $a_*(\tau)$, $\tau=0,1$, using Proposition~\ref{P2.6}. Then $C(n,k;q)$ is given by \eqref{2.9-2}, and the result is Theorem~\ref{T1.1}.
\end{rmk}

%%%%%%%%%%%%%%%%%%%%%%%%%%%%%%%%%%%%%%%%%%%
% section 3
%%%%%%%%%%%%%%%%%%%%%%%%%%%%%%%%%%%%%%%%%%%

\section{Exceptional Primes}

We follow the notation of Section~2. First note that for $(u,v),(u'v')\in\mathcal S_k$, $S(u,v)=S(u',v')$ if and only if ${\overline f}_{(u,v),(u'v')}(\omega)=0$, where
\begin{align}\label{3.1}
f_{(u,v),(u'v')}\,&=X^{u+v'}-X^{u'+v}-X^u+X^{u'}+X^v-X^{v'}\cr
&=(1-X^u)(1-X^{v'})-(1-X^{u'})(1-X^v)\in\Bbb Z[X],
\end{align}
and $\overline{(\ )}$ is the reduction from $\Bbb Z[X]$ to $\Bbb Z_r[X]$. Let $\Phi_k\in\Bbb Z[X]$ denote the $k$th cyclotomic polynomial. Then
\[
\overline{\Phi}_k=\prod_{l\in\Bbb Z_k^\times}(X-\omega^l),
\]
where $\Bbb Z_k^\times$ is the multiplicative group of $\Bbb Z_k$.

\begin{lem}\label{L3.1}
Let $(u,v),(u',v')\in\mathcal S_k$ be such that $(u,v)\ne(u',v')$. Then 
\[
f_{(u,v),(u'v')}\not\equiv 0 \pmod{\Phi_k}.
\]
\end{lem}

\begin{proof} Assume to the contrary that $f_{(u,v),(u'v')}\equiv 0 \pmod{\Phi_k}$. Then
\begin{equation}\label{3.2}
(\zeta_k^u-1)(\zeta_k^{v'}-1)=(\zeta_k^{u'}-1)(\zeta_k^v-1),
\end{equation}
where $\zeta_k=e^{2\pi i/k}$. We treat $u,v,u',v'$ as elements of $\{1,2,\dots,k-1\}$. The polar decomposition of $\zeta_k^u-1$ is
\[
\zeta_k^u-1=2\sin\frac{\pi u}k \cdot e^{i(\frac{\pi u}k+\frac \pi 2)};
\]
see Figure~\ref{F1}. Then \eqref{3.2} is equivalent to the following system:
\begin{align}\label{3.3}
&\sin\frac{\pi u}k\sin\frac{\pi v'}k=\sin\frac{\pi u'}k\sin\frac{\pi v}k,\\
\label{3.4}
&\frac{\pi u}k+\frac{\pi v'}k=\frac{\pi u'}k+\frac{\pi v}k.
\end{align}

%%%%%%% Fig 1 %%%%%%%%%%%%%%%%%%%%%%%%%%%%%
\begin{figure}
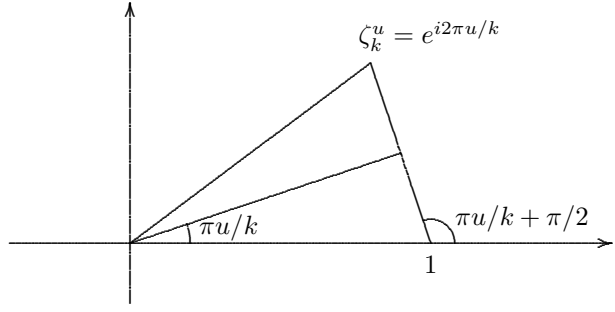

\vskip5mm
\[
\beginpicture
\setcoordinatesystem units <8mm,8mm> point at 0 0
\arrow <5pt> [.2,.67] from -2 0 to 8 0
\arrow <5pt> [.2,.67] from 0 -1 to 0 4
\circulararc 20 degrees from 1 0 center at 0 0
\circulararc 110 degrees from 5.4 0 center at 5 0

\setlinear 
\plot 0 0  4 3  5 0 /
\plot 0 0  4.5 1.5 /

\put {$\zeta_k^u=e^{i2\pi u/k}$} [bl] at 3.8 3.2
\put {$\pi u/k$} [bl] at 1.15 0.05
\put {$1$} [t] at 5 -0.2
\put {$\pi u/k+\pi/2$} [bl] at 5.4 0.2
\endpicture
\]
\caption{\label{F1} Polar decomposition of $\zeta_k^u-1$}
\end{figure}
%%%%%%% end Fig 1 %%%%%%%%%%%%%%%%%%%%%%%%%

\noindent
By \eqref{3.3},
\[
\cos\frac\pi k(u-v')-\cos\frac\pi k(u+v')=\cos\frac\pi k(u'-v)-\cos\frac\pi k(u'+v).
\]
Thus by \eqref{3.4},
\[
\cos\frac\pi k(u-v')=\cos\frac\pi k(u'-v),
\]
i.e., 
\begin{equation}\label{3.5}
u-v'=\pm(u'-v).
\end{equation}
Combining \eqref{3.4} and \eqref{3.5} gives
\[
\begin{cases}
u+v'=u'+v,\cr
u-v'=\pm(u'-v).
\end{cases}
\]
It follows that either $(u,v)=(u',v')$ or $(u,u')=(v,v')$, neither of which is possible.
\end{proof}

Let $\mathcal P_k$ denote the set of primes $r>k+1$ such that $r\equiv 1\pmod k$ and $S:\mathcal S_k\to\Bbb Z_r^*\setminus\{-1\}$ is not one-to-one. Recall that the elements of $\mathcal P_k$ are called exceptional primes. Moreover, for each $m\in\Bbb Z$, let $\mathcal P_k(m)$ be the set of primes divisors $r$ of $m$ such that $r>k+1$ and $r\equiv 1\pmod k$.

\begin{thm}\label{T3.2}
We have 
\begin{equation}\label{3.6}
\mathcal P_k=\mathcal P_k\biggl(\,\prod_{\substack{(u,v),(u',v')\in\mathcal S_k\cr (u,v)\ne(u',v')}}\text{\rm Res}(f_{(u,v),(u',v')},\Phi_k)\biggr),
\end{equation}
where $\text{\rm Res}(\cdot\, ,\,\cdot)$ is the resultant.
\end{thm}

\begin{rmk}\label{R3.3}\rm
By Lemma~\ref{L3.1}, the product in \eqref{3.6} is a nonzero integer. Hence it follows from \eqref{3.6} that $|\mathcal P_k|<\infty$.
\end{rmk}

\begin{proof}[Proof of Theorem~\ref{T3.2}] Let $\mathcal R$ denote the right side of \eqref{3.6}.

First assume that $r\in\mathcal P_k$. Then there exist $(u,v),(u',v')\in\mathcal S_k$ such that $(u,v)\ne (u',v')$ and $\overline{f}_{(u,v),(u',v')}(\omega)=0$, where $\overline{(\ )}$ is the reduction from $\Bbb Z[X]$ to $\Bbb Z_r[X]$. Then
\[
\overline{\text{Res}(f_{(u,v),(u',v')},\Phi_k)}=\text{Res}(\overline{f}_{(u,v),(u',v')},\overline{\Phi}_k)=0,
\]
and hence $r\in\mathcal R$.

Next assume that $r\in\mathcal R$. Then there exist $(u,v),(u',v')\in\mathcal S_k$ such that $(u,v)\ne(u',v')$ and $\text{Res}(\overline{f}_{(u,v),(u',v')},\overline{\Phi}_k)=0$. It follows that $\overline{f}_{(u,v),(u',v')}(\omega^l)=0$ for some $l\in\Bbb Z_k^\times$. Then $\overline{f}_{(lu,lv),(lu',lv')}(\omega)=0$, where $(lu,lv),(lu',lv')\in\mathcal S_k$ and $(lu,lv)\ne(lu',lv')$. 
\end{proof}

Since $f_{(u,v),(u',v')}=-f_{(u',v'),(u,v)}=-f_{(v,u),(v',u')}$, for the product in \eqref{3.6}, we only have to consider those $(u,v),(u',v')\in\mathcal S_k$ such that $u\le\min\{v,u',v'\}$. Moreover, we may further exclude those $(u,v), (u',v')$ with $u=u'$ or $v=v'$. In fact, if say $v=v'$, then for any prime $r\equiv 1\pmod k$, $\overline{f}_{(u,v),(u',v')}(\omega^l)=(\omega^{lu'}-\omega^{lu})(1-\omega^{lv})$ is never $0$ for all $l\in\Bbb Z_k^\times$. Hence
\begin{equation}\label{3.7}
\text{Res}(f_{(u,v),(u',v')},\Phi_k)\not\equiv 0\pmod r.
\end{equation}
If one prefers a more direct argument for \eqref{3.7}, observe that
\begin{align}\label{3.8}
&\text{Res}(f_{(u,v),(u',v')},\Phi_k)\cr
=\,&\text{Res}\bigl(-X^u(X^{u'-u}-1)(X^v-1),\Phi_k\bigr)\cr
=\,&\pm\prod_{l\in\Bbb Z_k^\times}(\zeta_k^{(u'-u)l}-1)(\zeta_k^{vl}-1)\cr
=\,&\pm\Phi_{k/\text{gcd}(k,u'-u)}(1)^{\text{gcd}(k,u'-u)}\Phi_{k/\text{gcd}(k,v)}(1)^{\text{gcd}(k,v)}
\end{align}
For each integer $m>1$, $\Phi_m\mid 1+X+\cdots+X^{m-1}$ and $\Phi_m(1)\mid m$. Therefore by \eqref{3.8}, $\text{Res}(f_{(u,v),(u',v')},\Phi_k)$ does not have any prime divisor $>k$.

We conclude that 
\begin{equation}\label{3.9}
\mathcal P_k=\mathcal P_k\biggl(\,\prod_{\substack{(u,v),(u',v')\in\mathcal S_k\cr v,u'>u,\, v'\ge u,\, v\ne v'}}\text{Res}(f_{(u,v),(u',v')},\Phi_k)\biggr).
\end{equation}
When $k\le 3$, $|\mathcal S_k|\le 1$, so the product in \eqref{3.9} is an empty one and hence $\mathcal P_k=\emptyset$. The sets $\mathcal P_k$, $4\le k\le 20$, are given in Table~\ref{Tb1} in the next section.

%%%%%%%%%%%%%%%%%%%%%%%%%%%%%%%%%%%%%%%%%%%
%   section 4
%%%%%%%%%%%%%%%%%%%%%%%%%%%%%%%%%%%%%%%%%%%

\section{Computation of $C(n,k;q)$ for Exceptional Primes}

\begin{algo}\label{A4.1}\rm
Given an integer $k>0$ and an exceptional prime $r=nk+1\in\mathcal P_k$, the following steps produce the complexity $C(n,k;q)$. 
\begin{labeling}{\hspace{13mm}}
\item[{\bf Step 1.}]
Find $\omega\in\Bbb Z_r^*$ with $o(\omega)=k$. Compute the multiset $\frak S=\{S(u,v):(u,v)\in\mathcal S_k\}$. Each element of $\frak S$ is an integer in $\{1,2,\dots,r-2\}$.

\medskip

\item[{\bf Step 2.}]
Tally the multiset $\frak S$ to determine the multiset $\frak M$ of the multiplicities of the elements of $\frak S$ and the multiset $\frak M_*$ of the multiplicities of the elements in $\frak S$ that belong to $(-1)^k\langle \omega\rangle$. More precisely, if 
\[
\frak S=\bigl\{\underbrace{x_1,\dots,x_1}_{m_1},\dots,\underbrace{x_a,\dots,x_a}_{m_a},\dots,\underbrace{x_b,\dots,x_b}_{m_b}\bigr\},
\]
where $x_1,\dots,x_b$ are distinct and $\{x_1,\dots,x_b\}\cap(-1)^k\langle\omega\rangle=\{x_1,\dots,x_a\}$, then $\frak M=\{m_1,\dots,m_b\}$ and $\frak M_*=\{m_1,\dots,m_a\}$.

\medskip

\item[{\bf Step 3.}] 
For each $2\le \tau\le k$, compute
\begin{gather}\label{4.1}
a(\tau)=\frac 1\tau\cdot(\text{the multiplicity of $\tau-1$ in $\frak M$}),\\
\label{4.2}
a_*(\tau)=\frac 1\tau\cdot(\text{the multiplicity of $\tau-1$ in $\frak M_*$}).
\end{gather}
Also compute
\begin{gather}\label{4.3}
a(1)=nk-1-\sum_{\tau=2}^k\tau a(\tau),\qquad a(0)=n^2-\sum_{\tau=1}^ka(\tau),\\
\label{4.4}
a_*(1)=k-1-\sum_{\tau=2}^k\tau a_*(\tau),\qquad a_*(0)=n-\sum_{\tau=1}^ka_*(\tau).
\end{gather}

\medskip

\item[{\bf Step 4.}]
Compute 
\begin{equation}\label{4.5}
C(n,k;q)=n^2-\sum_{\tau\equiv 0\,(\text{mod}\, p)}\bigl(a(\tau)-a_*(\tau)\bigr)-\sum_{\tau\equiv k\,(\text{mod}\, p)}a_*(\tau).
\end{equation}
\end{labeling}
\end{algo}

\noindent{\bf Note.} In the above, \eqref{4.1} -- \eqref{4.4} are from Proposition~\ref{P2.6}, and \eqref{4.5} is \eqref{2.9-2}. The algorithm works for all primes $r=nk+1$, but it is not necessary if $r\notin\mathcal P_k$ because of Theorem~\ref{T1.1}.

The above algorithm is applied to $4\le k\le 20$ and $r\in\mathcal P_k$. The sequences $a(\tau)$ and $a_*(\tau)$ are given in Table~\ref{Tb1}. Step~4 is easy computation but the results are too lengthy to be included. Instead, we use examples to exhibit the formulas for $C(n,k;q)$ for a few pairs of parameters $(n,k)$.

\begin{exmp}\label{E4.1}\rm
Let $k=6$, $n=3$, $r=19$. The nonzero terms of the sequences $a(\tau)$ and $a_*(\tau)$ are
\begin{align*}
&a(1)=3,\quad a(2)=4,\quad a(3)=2,\cr
&a_*(1)=1,\quad a_*(2)=2.
\end{align*}
By \eqref{4.5},
\begin{equation}\label{4.6}
C(3,6;q)=
\begin{cases}
5&\text{if}\ p=2,\cr
7&\text{if}\ p=3,\cr
8&\text{if}\ p=5,\cr
9&\text{if}\ p>5.
\end{cases}
\end{equation}
Note that \eqref{4.6} is different from the result in \cite[Theorem~3.14]{Christopoulou-Garefalakis-Panario-Thomson-DCC-2012}. It appears that the exceptional prime $r=19$ for $k=6$ was not detected in \cite{Christopoulou-Garefalakis-Panario-Thomson-DCC-2012}.
\end{exmp}

\begin{exmp}\label{E4.2}\rm
Let $k=7$, $n=6$, $r=43$. The nonzero terms of the sequences $a(\tau)$ and $a_*(\tau)$ are
\begin{align*}
&a(0)=9,\quad a(1)=14,\quad a(2)=12,\quad a(3)=1,\cr
&a_*(0)=2,\quad a_*(1)=2,\quad a_*(2)=2.
\end{align*}
By \eqref{4.5},
\begin{equation}\label{4.7}
C(6,7;q)=
\begin{cases}
26&\text{if}\ p=3,\cr
27&\text{if}\ p=5,7,\cr
29&\text{if}\ p>7.
\end{cases}
\end{equation}
The prime $p=2$ is not included in \eqref{4.7} since the order of $2$ in $\Bbb Z_{43}^*$ is $14$ and $\text{gcd}(nk/14,n)\ne 1$. It appears that the exceptional prime $r=43$ for $k=7$ was not detected in \cite[Theorem~10]{Liao-Hu-AMS-2014}.
\end{exmp}

\begin{exmp}\label{E4.3}\rm
Let $k=20$, $n=1166$, $r=23321$. The nonzero terms of the sequences $a(\tau)$ and $a_*(\tau)$ are
\begin{align*}
&a(0)=1336402,\quad a(1)=22995,\quad a(2)=153,\quad a(3)=6,\cr
&a_*(0)=1156,\quad a_*(1)=1,\quad a_*(2)=9.
\end{align*}
By \eqref{4.5},
\begin{equation}\label{4.8}
C(1166,20;q)=
\begin{cases}
24295&\text{if}\ p=3,\cr
24310&\text{if}\ p=17\ \text{or}\ p>19.
\end{cases}
\end{equation}
The primes $p=2,5,7,11,13,19$ are not included in \eqref{4.8} since $\text{gcd}(nk/e(p,r),n)\ne 1$ for these primes $p$, where $e(p,r)$ is the order of $p$ in $\Bbb Z_r^*$. \end{exmp}

Table~\ref{Tb1} suggests that for every prime $r=2k+1$, $r\in\mathcal P_k$. Moreover, in this case, the nonzero terms of $a(\tau)$ and $a_*(\tau)$ are 
\begin{equation}\label{4.5-1}
\begin{cases}
a(k/2-1)=1,\quad a(k/2)=3,\cr
a_*(k/2-1)=1,\quad a_*(k/2)=1
\end{cases}\quad\text{if $k$ is even},
\end{equation}
\begin{equation}\label{4.5-2}
\begin{cases}
a((k-1)/2)=3,\quad a((k+1)/2)=1,\cr
a_*((k-1)/2)=2
\end{cases}\quad\text{if $k$ is odd}.
\end{equation}
The above are indeed correct formulas and they follow from \cite[\S6]{Dickson-AJM-1935}. The following is a proof using the notation of the present paper.

Let $(u,v)\in\mathcal S_k$ and $z\in\Bbb Z_r\setminus\{0,-1\}$. Write $\omega^u=x^2$ and $\omega^v=y^2$, where $x,y\in\Bbb Z_r\setminus\{0,\pm1\}$. Then $S(u,v)=z$ if and only if
\[
\frac{x^2-1}{y^2-1}=1+z^{-1},
\]
i.e.,
\begin{equation}\label{4.5-3}
x^2-\lambda y^2=1-\lambda,
\end{equation}
where $\lambda=1+z^{-1}\in\Bbb Z_r\setminus\{0,1\}$. By \cite[Theorem~6.26]{Lidl-Niederreiter-1997},
\[
|\{(x,y)\in\Bbb Z_r^2:x^2-\lambda y^2=1-\lambda\}|=r-\eta(\lambda),
\]
where $\eta$ is the quadratic character of $\Bbb Z_r$. It follows that 
\begin{align}\label{4.5-4}
&\bigl|\bigl\{(x,y): x,y\in\Bbb Z_r\setminus\{0,\pm1\},\ x^2-\lambda y^2=1-\lambda\bigr\}\bigr|\cr
=\,&r-\eta(\lambda)-4-(1+\eta(1-\lambda))-(1+\eta(1-\lambda^{-1}))\cr
=\,&r-6-\eta(\lambda)-\eta(1-\lambda)-\eta(1-\lambda^{-1}).
\end{align}

First assume that $k$ is even. Then $\eta(\lambda)+\eta(1-\lambda)+\eta(1-\lambda^{-1})=-1$ or $3$ since the left side of \eqref{4.5-4} is $\equiv 0\pmod 4$. Let
\[
b_i=\bigl|\big\{\lambda\in\Bbb Z_r\setminus\{0,1\}:\eta(\lambda)+\eta(1-\lambda)+\eta(1-\lambda^{-1})=i\bigr\}\bigr|,\quad i=-1,3.
\]
Then
\begin{equation}\label{4.5-5}
b_{-1}+b_3=r-2
\end{equation}
and    
\begin{equation}\label{4.5-6}
-b_{-1}+3b_3=\sum_{\lambda\in\Bbb Z_r\setminus\{0,1\}}\bigl(\eta(\lambda)+\eta(1-\lambda)+\eta(1-\lambda^{-1})\bigr)=-3.
\end{equation}
Combining \eqref{4.5-5} and \eqref{4.5-6} gives
\[
b_{-1}=\frac 34(r-1),\quad b_3=\frac 14(r-5).
\]
For $i=-1,3$, by \eqref{4.5-4},
\[
\Bigl|\Bigl\{z\in\Bbb Z_r\setminus\{0,-1\}:|S^{-1}(z)|=\frac 14(r-6-i)\Bigr\}\Bigr|=b_i.
\]
Hence
\[
a((r-6-i)/4+1)=\frac 1{(r-6-i)/4+1}b_i,
\]
i.e., $a(k/2-1)+1)=1$ and $a(k/2)=3$.

To determine $a_*((r-6-i)/4+1)$ with $i=-1,3$, let
\begin{align*}
b_i^*\,&=\bigl|\bigl\{\lambda\in\Bbb Z_r\setminus\{0,1\}:\eta(\lambda-1)=1,\ \eta(\lambda)+\eta(1-\lambda)+\eta(1-\lambda^{-1})=i\bigr\}\bigr|\cr
&=\bigl|\bigl\{\lambda\in\Bbb Z_r\setminus\{0,1\}:\eta(\lambda-1)=1,\ \eta(\lambda)=(i-1)/2\bigr\}\bigr|.
\end{align*}
In the above, let $1-\lambda=x^2$ and $\lambda=\epsilon y^2$, where $x,y\in\Bbb Z_r^*$ and   $\epsilon\in \Bbb Z$ is a fixed element such that $\eta(\epsilon)=(i-1)/2$. Then
\begin{align*}
b_i^*\,&=\frac 14|\{(x,y)\in\Bbb Z_r^2:x^2\ne 0,1,\ x^2+\epsilon y^2=1\}|\cr
&=\frac 14\bigl(r-\eta(-\epsilon)-(1+\eta(\epsilon))-2\bigr)\kern2cm \text{(by \cite[Theorem~6.26]{Lidl-Niederreiter-1997})}\cr
&=\frac 14(r-3-2\eta(\epsilon))\cr
&=\frac 14(r-2-i).
\end{align*}
Thus
\[
a_*((r-6-i)/4+1)=\frac1{(r-6-i)/4+1}b_i^*=1,
\]
i.e., $a_*(k/2-1)=a_*(k/2)=1$.

Next assume that $k$ is odd. Then in \eqref{4.5-4}, $\eta(\lambda)+\eta(1-\lambda)+\eta(1-\lambda^{-1})=-3$ or $1$. Let 
\[
c_i=\bigl|\bigl\{\lambda\in\Bbb Z_r\setminus\{0,1\}:\eta(\lambda)+\eta(1-\lambda)+\eta(1-\lambda^{-1})=i\bigr\}\bigr|,\quad i=-3,1.
\]
Then
\[
\begin{cases}
c_{-3}+c_1=r-2,\cr
-3c_{-3}+c_1=-3,
\end{cases}
\]
which gives 
\[
c_{-3}=\frac 14(r+1),\quad c_1=\frac 34(r-3).
\]
We have
\[
a((r-6-i)/4+1)=\frac 1{(r-6-i)/4+1}c_i,
\]
i.e., $a((k-1)/2)=3$, $a((k+1)/2)=1$. To determine $a_*((r-6-i)/4+1)$ with $i=-3,1$, let
\[
c_i^*=\bigl|\bigl\{\lambda\in\Bbb Z_r\setminus\{0,1\}:\eta(\lambda-1)=-1,\ \eta(\lambda)+\eta(1-\lambda)+\eta(1-\lambda^{-1})=i\bigr\}\bigr|.
\]
Note that $\eta(\lambda-1)=-1$ implies that $\eta(\lambda)+\eta(1-\lambda)+\eta(1-\lambda^{-1})=1$. Hence
\[
c_i^*=
\begin{cases}
0&\text{if}\ i=-3,\vspace{2mm}\cr
|\{\lambda\in\Bbb Z_r\setminus\{0,1\}:\eta(\lambda-1)=-1\}|=\displaystyle\frac{r-3}2&\text{if}\ i=1.
\end{cases}
\]
Thus
\[
a_*((r-6-i)/4+1)=\frac1{(r-6-i)/4+1}c_i^*,
\]
i.e., $a_*((k-1)/2)=2$ and $a_*((k+1)/2)=0$.

%%%%% Table 1 %%%%%%%%%%%%%%%%%%%%%%%%%%%%%
\begin{table}
\caption{Values of $a(\tau)$ and $a_*(\tau)$ for $4\le k\le 20$, $r=nk+1\in\mathcal P_k$}\label{Tb1}
\vskip-5mm
\[
\begin{tabular}{cc|ccccccccccc}
\multicolumn{13}{l}{The top sequence is $a(\tau)$, and the bottom one is $a_*(\tau)$, $0\le \tau\le 10$; only nonzero}\\ 
\multicolumn{13}{l}{terms are listed}\vspace{2mm}\\ 
\hline
$k$ & $n$ &\kern4mm  0\kern4mm  &\kern3mm  1\kern3mm  &\kern3mm  2\kern3mm  &\kern2mm 3\kern2mm &\kern2mm 4\kern2mm &\kern2mm 5\kern2mm &\kern2mm 6\kern2mm &\kern2mm 7\kern2mm &\kern2mm 8\kern2mm &\kern2mm 9\kern2mm &\kern2mm 10\kern2mm \\ \hline 
\hline
4 & -- & \\ \hline
5&2&&&3&1&&\\
&&&&2&&&\\ \hline
6&2&&&1&3&&&\\
&&&&1&1&&&\\ \hline
6&3&&3&4&2&&&\\
&&&1&2&&&&\\ \hline
7&4&1&6&6&3&&&&\\
&&&2&2&&&&&\\ \hline
7&6&9&14&12&1&&&&\\
&&2&2&2&&&&&\\ \hline
8&2&&&&1&3&&&&\\
&&&&&1&1&&&&\\ \hline
8&5&4&6&12&3&&&&&\\
&&2&&2&1&&&&&\\ \hline
9&2&&&&&3&1&&&&\\
&&&&&&2&&&&&\\ \hline
9&4&&3&10&&3&&&&&\\
&&&&4&&&&&&&\\ \hline
9&8&15&33&10&6&&&&&&\\
&&&8&&&&&&&&\\ \hline
9&12&62&60&19&3&&&&&&\\
&&6&4&2&&&&&&&\\ \hline
9&14&96&78&19&3&&&&&&\\
&&8&4&2&&&&&&&\\ \hline
9&30&657&219&22&2&&&&&&\\
&&22&8&&&&&&&&\\ \hline
10&3&&&3&1&5&&&&&&\\
&&&&1&1&1&&&&&&\\ \hline
10&4&1&&9&3&3&&&&&&\\
&&1&&1&1&1&&&&&&\\ \hline
10&6&4&12&15&3&2&&&&&&\\
&&2&&3&1&&&&&&&\\ \hline
10&7&10&15&18&6&&&&&&&\\
&&2&1&4&&&&&&&&\\ \hline
10&10&31&45&18&6&&&&&&&\\
&&5&1&4&&&&&&&&\\ \hline
11&2&&&&&&3&1&&&&\\
&&&&&&&2&&&&&\\ \hline
11&6&1&15&15&&5&&&&&&\\
&&&2&4&&&&&&&&\\ \hline
11&8&13&21&27&&3&&&&&&\\
&&2&2&4&&&&&&&&\\ \hline
\end{tabular}
\]
\end{table}
%%%%%%%%%%%%%%%%%%%%%%%%%%%%%%%%%%%%%%%%%%%

\addtocounter{table}{-1}
\begin{table}
\caption{continued}%\label{Tb1}
\vskip-5mm
\[
\begin{tabular}{cc|ccccccccccc}
\hline
$k$ & $n$ &\kern4mm  0\kern4mm  &\kern3mm  1\kern3mm  &\kern3mm  2\kern3mm  &\kern2mm 3\kern2mm &\kern2mm 4\kern2mm &\kern2mm 5\kern2mm &\kern2mm 6\kern2mm &\kern2mm 7\kern2mm &\kern2mm 8\kern2mm &\kern2mm 9\kern2mm &\kern2mm 10\kern2mm \\ \hline  
\hline
11&18&169&116&36&3&&&&&&&\\
&&10&6&2&&&&&&&&\\ \hline
11&32&712&279&27&6&&&&&&&\\
&&22&10&&&&&&&&&\\ \hline
11&36&943&314&36&3&&&&&&&\\
&&28&6&2&&&&&&&&\\ \hline
11&62&3207&594&42&1&&&&&&&\\
&&54&6&2&&&&&&&&\\ \hline
12&3&&&1&2&3&3&&&&&\\
&&&&1&&1&1&&&&&\\ \hline
12&5&3&&10&9&3&&&&&&\\
&&1&&2&1&1&&&&&&\\ \hline
12&6&5&6&10&15&&&&&&&\\
&&1&&4&1&&&&&&&\\ \hline
12&8&12&21&19&12&&&&&&&\\
&&2&1&5&&&&&&&&\\ \hline
12&9&23&15&37&6&&&&&&&\\
&&3&1&5&&&&&&&&\\ \hline
12&13&63&63&37&6&&&&&&&\\
&&7&1&5&&&&&&&&\\ \hline
12&15&99&75&49&2&&&&&&&\\
&&9&1&5&&&&&&&&\\ \hline
12&16&114&99&37&6&&&&&&&\\
&&10&1&5&&&&&&&&\\ \hline
13&4&&&6&3&6&&1&&&&\\
&&&&2&&2&&&&&&\\ \hline
13&6&1&11&12&6&6&&&&&&\\
&&&2&2&2&&&&&&&\\ \hline
13&10&22&39&30&6&3&&&&&&\\
&&2&4&4&&&&&&&&\\ \hline
13&12&40&63&36&&5&&&&&&\\
&&4&4&4&&&&&&&&\\ \hline
13&24&322&203&48&&3&&&&&&\\
&&16&4&4&&&&&&&&\\ \hline
13&40&1144&396&57&3&&&&&&&\\
&&30&8&2&&&&&&&&\\ \hline
13&46&1579&483&48&6&&&&&&&\\
&&34&12&&&&&&&&&\\ \hline
13&66&3559&743&48&6&&&&&&&\\
&&54&12&&&&&&&&&\\ \hline
13&100&8761&1185&48&6&&&&&&&\\
&&88&12&&&&&&&&&\\ \hline
13&124&13828&1488&57&3&&&&&&&\\
&&114&8&2&&&&&&&&\\ \hline
13&154&21778&1878&57&3&&&&&&&\\
&&144&8&2&&&&&&&&\\ \hline
\end{tabular}
\]
\end{table}

%%%%%%%%%%%%%%%%%%%%%%%%%%%%%%%%%%%%%%%%%%%

\addtocounter{table}{-1}
\begin{table}
\caption{continued}%\label{Tb1}
\vskip-5mm
\[
\begin{tabular}{cc|ccccccccccc}
\hline
$k$ & $n$ &\kern4mm  0\kern4mm  &\kern3mm  1\kern3mm  &\kern3mm  2\kern3mm  &\kern2mm 3\kern2mm &\kern2mm 4\kern2mm &\kern2mm 5\kern2mm &\kern2mm 6\kern2mm &\kern2mm 7\kern2mm &\kern2mm 8\kern2mm &\kern2mm 9\kern2mm &\kern2mm 10\kern2mm \\ \hline  
\hline
13&210&41436&2600&63&1&&&&&&&\\
&&200&8&2&&&&&&&&\\ \hline
14&2&&&&&&&1&3&&&\\
&&&&&&&&1&1&&&\\ \hline
14&3&&&&1&5&&3&&&&\\
&&&&&1&1&&1&&&&\\ \hline
14&5&1&&12&6&3&3&&&&&\\
&&1&&2&&1&1&&&&&\\ \hline
14&8&7&24&15&15&3&&&&&&\\
&&3&&3&1&1&&&&&&\\ \hline
14&9&18&14&39&7&3&&&&&&\\
&&4&&3&1&1&&&&&&\\ \hline
14&14&67&75&42&12&&&&&&&\\
&&7&1&6&&&&&&&&\\ \hline
14&15&82&87&48&6&2&&&&&&\\
&&8&1&6&&&&&&&&\\ \hline
14&17&118&114&51&3&3&&&&&&\\
&&12&&3&1&1&&&&&&\\ \hline
14&20&187&159&42&12&&&&&&&\\
&&13&1&6&&&&&&&&\\ \hline
14&24&313&197&60&6&&&&&&&\\
&&17&1&6&&&&&&&&\\ \hline
14&27&424&239&60&6&&&&&&&\\
&&20&1&6&&&&&&&&\\ \hline
14&35&808&351&60&6&&&&&&&\\
&&28&1&6&&&&&&&&\\ \hline
14&39&1045&414&57&3&2&&&&&&\\
&&33&&5&1&&&&&&&\\ \hline
14&47&1624&519&60&6&&&&&&&\\
&&40&1&6&&&&&&&&\\ \hline
14&50&1873&561&60&6&&&&&&&\\
&&43&1&6&&&&&&&&\\ \hline
14&75&4648&911&60&6&&&&&&&\\
&&68&1&6&&&&&&&&\\ \hline
14&78&5068&944&69&3&&&&&&&\\
&&72&&5&1&&&&&&&\\ \hline
14&92&7249&1149&60&6&&&&&&&\\
&&85&1&6&&&&&&&&\\ \hline
15&2&&&&&&&&3&1&&\\
&&&&&&&&&2&&&\\ \hline
15&4&&&1&9&3&&3&&&&\\
&&&&&2&2&&&&&&\\ \hline
15&10&21&27&37&12&3&&&&&&\\
&&&6&4&&&&&&&&\\ \hline
15&12&38&48&46&9&3&&&&&&\\
&&4&2&6&&&&&&&&\\ \hline
\end{tabular}
\]
\end{table}

%%%%%%%%%%%%%%%%%%%%%%%%%%%%%%%%%%%%%%%%%%%

\addtocounter{table}{-1}
\begin{table}
\caption{continued}%\label{Tb1}
\vskip-5mm
\[
\begin{tabular}{cc|ccccccccccc}
\hline
$k$ & $n$ &\kern4mm  0\kern4mm  &\kern3mm  1\kern3mm  &\kern3mm  2\kern3mm  &\kern2mm 3\kern2mm &\kern2mm 4\kern2mm &\kern2mm 5\kern2mm &\kern2mm 6\kern2mm &\kern2mm 7\kern2mm &\kern2mm 8\kern2mm &\kern2mm 9\kern2mm &\kern2mm 10\kern2mm \\ \hline  
\hline
15&14&54&93&37&6&6&&&&&&\\
&&6&4&2&2&&&&&&&\\ \hline
15&16&93&99&55&6&3&&&&&&\\
&&6&6&4&&&&&&&&\\ \hline
15&18&125&147&37&12&3&&&&&&\\
&&8&6&4&&&&&&&&\\ \hline
15&22&234&180&63&6&&1&&&&&\\
&&14&4&2&2&&&&&&&\\ \hline
15&28&453&246&82&3&&&&&&&\\
&&16&10&2&&&&&&&&\\ \hline
15&36&833&399&55&6&3&&&&&&\\
&&26&6&4&&&&&&&&\\ \hline
15&38&954&423&55&12&&&&&&&\\
&&24&14&&&&&&&&&\\ \hline
15&42&1214&483&55&12&&&&&&&\\
&&28&14&&&&&&&&&\\ \hline
15&50&1839&576&82&3&&&&&&&\\
&&38&10&2&&&&&&&&\\ \hline
15&88&6513&1146&82&3&&&&&&&\\
&&76&10&2&&&&&&&&\\ \hline
15&92&7173&1206&82&3&&&&&&&\\
&&80&10&2&&&&&&&&\\ \hline
15&98&8214&1323&55&12&&&&&&&\\
&&84&14&&&&&&&&&\\ \hline
15&144&18662&1995&73&6&&&&&&&\\
&&130&14&&&&&&&&&\\ \hline
15&150&20336&2085&73&6&&&&&&&\\
&&136&14&&&&&&&&&\\ \hline
15&154&21492&2145&73&6&&&&&&&\\
&&140&14&&&&&&&&&\\ \hline
15&186&31892&2625&73&6&&&&&&&\\
&&172&14&&&&&&&&&\\ \hline
15&208&40230&2955&73&6&&&&&&&\\
&&194&14&&&&&&&&&\\ \hline
15&242&55020&3465&73&6&&&&&&&\\
&&228&14&&&&&&&&&\\ \hline
15&270&68940&3873&85&2&&&&&&&\\
&&256&14&&&&&&&&&\\ \hline
15&380&138786&5535&73&6&&&&&&&\\
&&366&14&&&&&&&&&\\ \hline
15&388&144810&5655&73&6&&&&&&&\\
&&374&14&&&&&&&&&\\ \hline
15&392&147870&5715&73&6&&&&&&&\\
&&378&14&&&&&&&&&\\ \hline
15&458&202980&6705&73&6&&&&&&&\\
&&444&14&&&&&&&&&\\ \hline
\end{tabular}
\]
\end{table}

%%%%%%%%%%%%%%%%%%%%%%%%%%%%%%%%%%%%%%%%%%%

\addtocounter{table}{-1}
\begin{table}
\caption{continued}%\label{Tb1}
\vskip-5mm
\[
\begin{tabular}{cc|ccccccccccc}
\hline
$k$ & $n$ &\kern4mm  0\kern4mm  &\kern3mm  1\kern3mm  &\kern3mm  2\kern3mm  &\kern2mm 3\kern2mm &\kern2mm 4\kern2mm &\kern2mm 5\kern2mm &\kern2mm 6\kern2mm &\kern2mm 7\kern2mm &\kern2mm 8\kern2mm &\kern2mm 9\kern2mm &\kern2mm 10\kern2mm \\ \hline  
\hline
15&3084&9464886&46083&85&2&&&&&&&\\
&&3070&14&&&&&&&&&\\ \hline
16&6&1&9&3&12&11&&&&&&\\
&&1&1&1&&3&&&&&&\\ \hline
16&7&4&6&21&12&3&3&&&&&\\
&&2&&3&&1&1&&&&&\\ \hline
16&12&37&44&42&21&&&&&&&\\
&&5&&6&1&&&&&&&\\ \hline
16&15&79&65&69&12&&&&&&&\\
&&7&1&7&&&&&&&&\\ \hline
16&16&90&90&66&7&3&&&&&&\\
&&10&&4&1&1&&&&&&\\ \hline
16&21&193&176&60&9&3&&&&&&\\
&&15&&4&1&1&&&&&&\\ \hline
16&22&226&177&69&12&&&&&&&\\
&&14&1&7&&&&&&&&\\ \hline
16&25&325&207&87&6&&&&&&&\\
&&17&1&7&&&&&&&&\\ \hline
16&27&397&239&87&6&&&&&&&\\
&&19&1&7&&&&&&&&\\ \hline
16&28&436&255&87&6&&&&&&&\\
&&20&1&7&&&&&&&&\\ \hline
16&36&820&383&87&6&&&&&&&\\
&&28&1&7&&&&&&&&\\ \hline
16&37&877&399&87&6&&&&&&&\\
&&29&1&7&&&&&&&&\\ \hline
16&40&1060&447&87&6&&&&&&&\\
&&32&1&7&&&&&&&&\\ \hline
16&55&2245&687&87&6&&&&&&&\\
&&47&1&7&&&&&&&&\\ \hline
16&76&4660&1023&87&6&&&&&&&\\
&&68&1&7&&&&&&&&\\ \hline
16&132&15412&1919&87&6&&&&&&&\\
&&124&1&7&&&&&&&&\\ \hline
16&133&15661&1935&87&6&&&&&&&\\
&&125&1&7&&&&&&&&\\ \hline
17&6&&5&15&&15&&1&&&&\\
&&&&4&&2&&&&&&\\ \hline
17&8&7&12&21&15&9&&&&&&\\
&&&4&2&&2&&&&&&\\ \hline
17&14&48&78&60&3&6&&1&&&&\\
&&6&4&2&&2&&&&&&\\ \hline
17&18&118&125&66&12&3&&&&&&\\
&&6&8&4&&&&&&&&\\ \hline
17&24&268&227&66&12&3&&&&&&\\
&&12&8&4&&&&&&&&\\ \hline
\end{tabular}
\]
\end{table}

%%%%%%%%%%%%%%%%%%%%%%%%%%%%%%%%%%%%%%%%%%%

\addtocounter{table}{-1}
\begin{table}
\caption{continued}%\label{Tb1}
\vskip-5mm
\[
\begin{tabular}{cc|ccccccccccc}
\hline
$k$ & $n$ &\kern4mm  0\kern4mm  &\kern3mm  1\kern3mm  &\kern3mm  2\kern3mm  &\kern2mm 3\kern2mm &\kern2mm 4\kern2mm &\kern2mm 5\kern2mm &\kern2mm 6\kern2mm &\kern2mm 7\kern2mm &\kern2mm 8\kern2mm &\kern2mm 9\kern2mm &\kern2mm 10\kern2mm \\ \hline  
\hline
17&26&331&267&66&6&6&&&&&&\\
&&16&6&2&2&&&&&&&\\ \hline
17&36&790&413&84&6&3&&&&&&\\
&&24&8&4&&&&&&&&\\ \hline
17&38&904&447&84&6&3&&&&&&\\
&&26&8&4&&&&&&&&\\ \hline
17&54&2104&719&84&6&3&&&&&&\\
&&42&8&4&&&&&&&&\\ \hline
17&56&2290&753&84&6&3&&&&&&\\
&&44&8&4&&&&&&&&\\ \hline
17&60&2680&839&66&12&3&&&&&&\\
&&48&8&4&&&&&&&&\\ \hline
17&80&5155&1137&102&6&&&&&&&\\
&&64&16&&&&&&&&&\\ \hline
17&84&5734&1227&90&&5&&&&&&\\
&&72&8&4&&&&&&&&\\ \hline
17&98&8047&1461&84&12&&&&&&&\\
&&82&16&&&&&&&&&\\ \hline
17&138&16807&2139&90&6&2&&&&&&\\
&&122&16&&&&&&&&&\\ \hline
17&168&25480&2639&102&&3&&&&&&\\
&&156&8&4&&&&&&&&\\ \hline
17&180&29455&2837&102&6&&&&&&&\\
&&164&16&&&&&&&&&\\ \hline
17&194&34453&3075&102&6&&&&&&&\\
&&178&16&&&&&&&&&\\ \hline
17&204&38263&3245&102&6&&&&&&&\\
&&188&16&&&&&&&&&\\ \hline
17&210&40648&3338&111&3&&&&&&&\\
&&196&12&2&&&&&&&&\\ \hline
17&216&43093&3467&84&12&&&&&&&\\
&&200&16&&&&&&&&&\\ \hline
17&344&112603&5625&102&6&&&&&&&\\
&&328&16&&&&&&&&&\\ \hline
17&500&241615&8277&102&6&&&&&&&\\
&&484&16&&&&&&&&&\\ \hline
17&546&288949&9059&102&6&&&&&&&\\
&&530&16&&&&&&&&&\\ \hline
17&644&403903&10725&102&6&&&&&&&\\
&&628&16&&&&&&&&&\\ \hline
17&708&489343&11813&102&6&&&&&&&\\
&&692&16&&&&&&&&&\\ \hline
17&920&830878&15408&111&3&&&&&&&\\
&&906&12&2&&&&&&&&\\ \hline
17&1046&1076452&17550&111&3&&&&&&&\\
&&1032&12&2&&&&&&&&\\ \hline
\end{tabular}
\]
\end{table}

%%%%%%%%%%%%%%%%%%%%%%%%%%%%%%%%%%%%%%%%%%%

\addtocounter{table}{-1}
\begin{table}
\caption{continued}%\label{Tb1}
\vskip-5mm
\[
\begin{tabular}{cc|ccccccccccc}
\hline
$k$ & $n$ &\kern4mm  0\kern4mm  &\kern3mm  1\kern3mm  &\kern3mm  2\kern3mm  &\kern2mm 3\kern2mm &\kern2mm 4\kern2mm &\kern2mm 5\kern2mm &\kern2mm 6\kern2mm &\kern2mm 7\kern2mm &\kern2mm 8\kern2mm &\kern2mm 9\kern2mm &\kern2mm 10\kern2mm \\ \hline  
\hline
17&1140&1280335&19157&102&6&&&&&&&\\
&&1124&16&&&&&&&&&\\ \hline
17&1484&2177146&24996&111&3&&&&&&&\\
&&1470&12&2&&&&&&&&\\ \hline
17&1548&2370106&26084&111&3&&&&&&&\\
&&1534&12&2&&&&&&&&\\ \hline
17&2054&4184113&34695&102&6&&&&&&&\\
&&2038&16&&&&&&&&&\\ \hline
17&2570&6561330&43452&117&1&&&&&&&\\
&&2556&12&2&&&&&&&&\\ \hline
18&2&&&&&&&&&1&3&\\
&&&&&&&&&&1&1&\\ \hline
18&4&&&3&&3&7&3&&&&\\
&&&&1&&1&1&1&&&&\\ \hline
18&6&3&&4&23&3&&3&&&&\\
&&1&&2&1&1&&1&&&&\\ \hline
18&7&3&6&16&12&9&3&&&&&\\
&&1&&4&&1&1&&&&&\\ \hline
18&9&12&12&31&17&9&&&&&&\\
&&2&&5&1&1&&&&&&\\ \hline
18&10&15&27&28&24&6&&&&&&\\
&&1&1&8&&&&&&&&\\ \hline
18&11&24&30&37&27&3&&&&&&\\
&&4&&5&1&1&&&&&&\\ \hline
18&15&72&57&76&20&&&&&&&\\
&&6&1&8&&&&&&&&\\ \hline
18&17&96&102&73&15&3&&&&&&\\
&&10&&5&1&1&&&&&&\\ \hline
18&21&182&159&82&18&&&&&&&\\
&&12&1&8&&&&&&&&\\ \hline
18&22&213&159&100&12&&&&&&&\\
&&13&1&8&&&&&&&&\\ \hline
18&24&269&195&100&12&&&&&&&\\
&&15&1&8&&&&&&&&\\ \hline
18&27&368&249&100&12&&&&&&&\\
&&18&1&8&&&&&&&&\\ \hline
18&29&444&285&100&12&&&&&&&\\
&&20&1&8&&&&&&&&\\ \hline
18&30&491&285&118&6&&&&&&&\\
&&21&1&8&&&&&&&&\\ \hline
18&32&579&321&118&6&&&&&&&\\
&&23&1&8&&&&&&&&\\ \hline
18&34&675&357&118&6&&&&&&&\\
&&25&1&8&&&&&&&&\\ \hline
18&41&1074&483&118&6&&&&&&&\\
&&32&1&8&&&&&&&&\\ \hline
\end{tabular}
\]
\end{table}

%%%%%%%%%%%%%%%%%%%%%%%%%%%%%%%%%%%%%%%%%%%

\addtocounter{table}{-1}
\begin{table}
\caption{continued}%\label{Tb1}
\vskip-5mm
\[
\begin{tabular}{cc|ccccccccccc}
\hline
$k$ & $n$ &\kern4mm  0\kern4mm  &\kern3mm  1\kern3mm  &\kern3mm  2\kern3mm  &\kern2mm 3\kern2mm &\kern2mm 4\kern2mm &\kern2mm 5\kern2mm &\kern2mm 6\kern2mm &\kern2mm 7\kern2mm &\kern2mm 8\kern2mm &\kern2mm 9\kern2mm &\kern2mm 10\kern2mm \\ \hline  
\hline
18&42&1142&492&127&3&&&&&&&\\
&&34&&7&1&&&&&&&\\ \hline
18&45&1346&555&118&6&&&&&&&\\
&&36&1&8&&&&&&&&\\ \hline
18&62&2859&861&118&6&&&&&&&\\
&&53&1&8&&&&&&&&\\ \hline
18&81&5234&1203&118&6&&&&&&&\\
&&72&1&8&&&&&&&&\\ \hline
18&121&12594&1923&118&6&&&&&&&\\
&&112&1&8&&&&&&&&\\ \hline
18&126&13743&2001&130&2&&&&&&&\\
&&117&1&8&&&&&&&&\\ \hline
18&127&13974&2031&118&6&&&&&&&\\
&&118&1&8&&&&&&&&\\ \hline
18&321&97398&5511&130&2&&&&&&&\\
&&312&1&8&&&&&&&&\\ \hline
19&10&15&15&48&12&9&&1&&&&\\
&&2&2&4&&2&&&&&&\\ \hline
19&12&28&38&54&15&9&&&&&&\\
&&2&6&2&&2&&&&&&\\ \hline
19&22&193&189&81&18&3&&&&&&\\
&&8&10&4&&&&&&&&\\ \hline
19&24&243&230&93&3&6&&1&&&&\\
&&14&6&2&&2&&&&&&\\ \hline
19&30&460&329&99&6&6&&&&&&\\
&&18&8&2&2&&&&&&&\\ \hline
19&34&643&399&99&12&3&&&&&&\\
&&20&10&4&&&&&&&&\\ \hline
19&40&979&495&117&6&3&&&&&&\\
&&26&10&4&&&&&&&&\\ \hline
19&58&2392&861&99&6&6&&&&&&\\
&&46&8&2&2&&&&&&&\\ \hline
19&78&4741&1217&117&6&3&&&&&&\\
&&64&10&4&&&&&&&&\\ \hline
19&82&5305&1293&117&6&3&&&&&&\\
&&68&10&4&&&&&&&&\\ \hline
19&84&5593&1349&99&12&3&&&&&&\\
&&70&10&4&&&&&&&&\\ \hline
19&94&7183&1539&99&12&3&&&&&&\\
&&80&10&4&&&&&&&&\\ \hline
19&100&8239&1635&117&6&3&&&&&&\\
&&86&10&4&&&&&&&&\\ \hline
19&108&9745&1805&99&12&3&&&&&&\\
&&94&10&4&&&&&&&&\\ \hline
19&118&11821&1977&117&6&3&&&&&&\\
&&104&10&4&&&&&&&&\\ \hline
\end{tabular}
\]
\end{table}

%%%%%%%%%%%%%%%%%%%%%%%%%%%%%%%%%%%%%%%%%%%

\addtocounter{table}{-1}
\begin{table}
\caption{continued}%\label{Tb1}
\vskip-5mm
\[
\begin{tabular}{cc|ccccccccccc}
\hline
$k$ & $n$ &\kern4mm  0\kern4mm  &\kern3mm  1\kern3mm  &\kern3mm  2\kern3mm  &\kern2mm 3\kern2mm &\kern2mm 4\kern2mm &\kern2mm 5\kern2mm &\kern2mm 6\kern2mm &\kern2mm 7\kern2mm &\kern2mm 8\kern2mm &\kern2mm 9\kern2mm &\kern2mm 10\kern2mm \\ \hline  
\hline
19&120&12259&2015&117&6&3&&&&&&\\
&&106&10&4&&&&&&&&\\ \hline
19&142&17614&2409&135&6&&&&&&&\\
&&124&18&&&&&&&&&\\ \hline
19&204&37882&3605&117&12&&&&&&&\\
&&186&18&&&&&&&&&\\ \hline
19&232&49558&4137&117&12&&&&&&&\\
&&214&18&&&&&&&&&\\ \hline
19&234&50449&4179&123&&5&&&&&&\\
&&220&10&4&&&&&&&&\\ \hline
19&240&53182&4289&117&12&&&&&&&\\
&&222&18&&&&&&&&&\\ \hline
19&244&55048&4347&135&6&&&&&&&\\
&&226&18&&&&&&&&&\\ \hline
19&258&61810&4613&135&6&&&&&&&\\
&&240&18&&&&&&&&&\\ \hline
19&358&121504&6531&117&12&&&&&&&\\
&&340&18&&&&&&&&&\\ \hline
19&360&122902&6567&123&6&2&&&&&&\\
&&342&18&&&&&&&&&\\ \hline
19&402&154108&7367&117&12&&&&&&&\\
&&384&18&&&&&&&&&\\ \hline
19&460&203005&8457&135&&3&&&&&&\\
&&446&10&4&&&&&&&&\\ \hline
19&472&213964&8679&135&6&&&&&&&\\
&&454&18&&&&&&&&&\\ \hline
19&492&232867&9050&144&3&&&&&&&\\
&&476&14&2&&&&&&&&\\ \hline
19&750&548392&13979&117&12&&&&&&&\\
&&732&18&&&&&&&&&\\ \hline
19&778&590650&14493&135&6&&&&&&&\\
&&760&18&&&&&&&&&\\ \hline
19&810&640858&15101&135&6&&&&&&&\\
&&792&18&&&&&&&&&\\ \hline
19&852&709864&15899&135&6&&&&&&&\\
&&834&18&&&&&&&&&\\ \hline
19&1620&2593768&30491&135&6&&&&&&&\\
&&1602&18&&&&&&&&&\\ \hline
19&1678&2783950&31593&135&6&&&&&&&\\
&&1660&18&&&&&&&&&\\ \hline
19&1834&3328858&34557&135&6&&&&&&&\\
&&1816&18&&&&&&&&&\\ \hline
19&2112&4420564&39839&135&6&&&&&&&\\
&&2094&18&&&&&&&&&\\ \hline
19&2128&4488100&40143&135&6&&&&&&&\\
&&2110&18&&&&&&&&&\\ \hline
\end{tabular}
\]
\end{table}

%%%%%%%%%%%%%%%%%%%%%%%%%%%%%%%%%%%%%%%%%%%

\addtocounter{table}{-1}
\begin{table}
\caption{continued}%\label{Tb1}
\vskip-5mm
\[
\begin{tabular}{cc|ccccccccccc}
\hline
$k$ & $n$ &\kern4mm  0\kern4mm  &\kern3mm  1\kern3mm  &\kern3mm  2\kern3mm  &\kern2mm 3\kern2mm &\kern2mm 4\kern2mm &\kern2mm 5\kern2mm &\kern2mm 6\kern2mm &\kern2mm 7\kern2mm &\kern2mm 8\kern2mm &\kern2mm 9\kern2mm &\kern2mm 10\kern2mm \\ \hline  
\hline
19&2532&6363064&47819&135&6&&&&&&&\\
&&2514&18&&&&&&&&&\\ \hline
19&2824&7921471&53358&144&3&&&&&&&\\
&&2808&14&2&&&&&&&&\\ \hline
19&2892&8308864&54659&135&6&&&&&&&\\
&&2874&18&&&&&&&&&\\ \hline
19&2970&8764621&56132&144&3&&&&&&&\\
&&2954&14&2&&&&&&&&\\ \hline
19&3352&11172364&63399&135&6&&&&&&&\\
&&3334&18&&&&&&&&&\\ \hline
19&3880&14980828&73431&135&6&&&&&&&\\
&&3862&18&&&&&&&&&\\ \hline
19&4402&19294117&83340&144&3&&&&&&&\\
&&4386&14&2&&&&&&&&\\ \hline
19&4852&23449864&91899&135&6&&&&&&&\\
&&4834&18&&&&&&&&&\\ \hline
19&5188&26816920&98283&135&6&&&&&&&\\
&&5170&18&&&&&&&&&\\ \hline
19&5638&31680073&106824&144&3&&&&&&&\\
&&5622&14&2&&&&&&&&\\ \hline
19&6888&47313823&130574&144&3&&&&&&&\\
&&6872&14&2&&&&&&&&\\ \hline
19&7372&54206464&139779&135&6&&&&&&&\\
&&7354&18&&&&&&&&&\\ \hline
19&9198&84428595&174458&150&1&&&&&&&\\
&&9182&14&2&&&&&&&&\\ \hline
20&2&&&&&&&&&&1&3\\
&&&&&&&&&&&1&1\\ \hline
20&3&&&&&&3&1&2&3&&\\
&&&&&&&1&1&&1&&\\ \hline
20&5&1&&&9&9&&6&&&&\\
&&1&&&1&1&&2&&&&\\ \hline
20&9&4&15&39&12&9&&&2&&&\\
&&2&1&3&&3&&&&&&\\ \hline
20&12&25&39&51&18&11&&&&&&\\
&&5&1&3&&3&&&&&&\\ \hline
20&14&43&63&63&18&9&&&&&&\\
&&7&1&3&&3&&&&&&\\ \hline
20&20&136&156&90&9&9&&&&&&\\
&&12&&6&1&1&&&&&&\\ \hline
20&21&163&167&81&30&&&&&&&\\
&&11&1&9&&&&&&&&\\ \hline
20&23&217&189&99&24&&&&&&&\\
&&13&1&9&&&&&&&&\\ \hline
20&26&310&231&117&18&&&&&&&\\
&&16&1&9&&&&&&&&\\ \hline
\end{tabular}
\]
\end{table}

%%%%%%%%%%%%%%%%%%%%%%%%%%%%%%%%%%%%%%%%%%%

\addtocounter{table}{-1}
\begin{table}
\caption{continued}%\label{Tb1}
\vskip-5mm
\[
\begin{tabular}{cc|ccccccccccc}
\hline
$k$ & $n$ &\kern4mm  0\kern4mm  &\kern3mm  1\kern3mm  &\kern3mm  2\kern3mm  &\kern2mm 3\kern2mm &\kern2mm 4\kern2mm &\kern2mm 5\kern2mm &\kern2mm 6\kern2mm &\kern2mm 7\kern2mm &\kern2mm 8\kern2mm &\kern2mm 9\kern2mm &\kern2mm 10\kern2mm \\ \hline  
\hline
20&27&328&284&105&6&3&3&&&&&\\
&&20&&5&&1&1&&&&&\\ \hline
20&30&454&309&123&12&2&&&&&&\\
&&20&1&9&&&&&&&&\\ \hline
20&32&538&351&117&18&&&&&&&\\
&&22&1&9&&&&&&&&\\ \hline
20&33&583&371&117&18&&&&&&&\\
&&23&1&9&&&&&&&&\\ \hline
20&35&679&411&117&18&&&&&&&\\
&&25&1&9&&&&&&&&\\ \hline
20&38&820&516&90&9&9&&&&&&\\
&&30&&6&1&1&&&&&&\\ \hline
20&41&1021&513&135&12&&&&&&&\\
&&31&1&9&&&&&&&&\\ \hline
20&44&1210&591&117&18&&&&&&&\\
&&34&1&9&&&&&&&&\\ \hline
20&47&1423&651&117&18&&&&&&&\\
&&37&1&9&&&&&&&&\\ \hline
20&51&1741&713&135&12&&&&&&&\\
&&41&1&9&&&&&&&&\\ \hline
20&53&1897&786&108&15&3&&&&&&\\
&&45&&6&1&1&&&&&&\\ \hline
20&59&2470&846&162&3&&&&&&&\\
&&50&&8&1&&&&&&&\\ \hline
20&60&2566&875&153&6&&&&&&&\\
&&50&1&9&&&&&&&&\\ \hline
20&65&3091&975&153&6&&&&&&&\\
&&55&1&9&&&&&&&&\\ \hline
20&66&3202&995&153&6&&&&&&&\\
&&56&1&9&&&&&&&&\\ \hline
20&68&3412&1086&108&15&3&&&&&&\\
&&60&&6&1&1&&&&&&\\ \hline
20&69&3541&1073&135&12&&&&&&&\\
&&59&1&9&&&&&&&&\\ \hline
20&74&4162&1155&153&6&&&&&&&\\
&&64&1&9&&&&&&&&\\ \hline
20&80&4954&1308&126&9&3&&&&&&\\
&&72&&6&1&1&&&&&&\\ \hline
20&81&5107&1295&153&6&&&&&&&\\
&&71&1&9&&&&&&&&\\ \hline
20&86&5842&1395&153&6&&&&&&&\\
&&76&1&9&&&&&&&&\\ \hline
20&87&5989&1433&135&12&&&&&&&\\
&&77&1&9&&&&&&&&\\ \hline
20&93&6943&1571&117&18&&&&&&&\\
&&83&1&9&&&&&&&&\\ \hline
\end{tabular}
\]
\end{table}

%%%%%%%%%%%%%%%%%%%%%%%%%%%%%%%%%%%%%%%%%%%

\addtocounter{table}{-1}
\begin{table}
\caption{continued}%\label{Tb1}
\vskip-5mm
\[
\begin{tabular}{cc|ccccccccccc}
\hline
$k$ & $n$ &\kern4mm  0\kern4mm  &\kern3mm  1\kern3mm  &\kern3mm  2\kern3mm  &\kern2mm 3\kern2mm &\kern2mm 4\kern2mm &\kern2mm 5\kern2mm &\kern2mm 6\kern2mm &\kern2mm 7\kern2mm &\kern2mm 8\kern2mm &\kern2mm 9\kern2mm &\kern2mm 10\kern2mm \\ \hline  
\hline
20&95&7291&1575&153&6&&&&&&&\\
&&85&1&9&&&&&&&&\\ \hline
20&104&8902&1755&153&6&&&&&&&\\
&&94&1&9&&&&&&&&\\ \hline
20&122&12610&2115&153&6&&&&&&&\\
&&112&1&9&&&&&&&&\\ \hline
20&152&20230&2715&153&6&&&&&&&\\
&&142&1&9&&&&&&&&\\ \hline
20&156&21382&2795&153&6&&&&&&&\\
&&146&1&9&&&&&&&&\\ \hline
20&161&22867&2895&153&6&&&&&&&\\
&&151&1&9&&&&&&&&\\ \hline
20&173&26629&3153&135&12&&&&&&&\\
&&163&1&9&&&&&&&&\\ \hline
20&228&47590&4235&153&6&&&&&&&\\
&&218&1&9&&&&&&&&\\ \hline
20&251&58141&4713&135&12&&&&&&&\\
&&241&1&9&&&&&&&&\\ \hline
20&254&59602&4755&153&6&&&&&&&\\
&&244&1&9&&&&&&&&\\ \hline
20&311&90667&5895&153&6&&&&&&&\\
&&301&1&9&&&&&&&&\\ \hline
20&333&104395&6335&153&6&&&&&&&\\
&&323&1&9&&&&&&&&\\ \hline
20&348&114310&6635&153&6&&&&&&&\\
&&338&1&9&&&&&&&&\\ \hline
20&366&126802&6995&153&6&&&&&&&\\
&&356&1&9&&&&&&&&\\ \hline
20&408&158470&7835&153&6&&&&&&&\\
&&398&1&9&&&&&&&&\\ \hline
20&434&179842&8355&153&6&&&&&&&\\
&&424&1&9&&&&&&&&\\ \hline
20&686&457042&13395&153&6&&&&&&&\\
&&676&1&9&&&&&&&&\\ \hline
20&1166&1336402&22995&153&6&&&&&&&\\
&&1156&1&9&&&&&&&&\\ \hline
\end{tabular}
\]
\end{table}

%%%%% end Table 1 %%%%%%%%%%%%%%%%%%%%%%%%%

%%%%%%%%%%%%%%%%%%%%%%%%%%%%%%%%%%%%%%%%%%%

%%%%%%%%%%%%%%%%%%%%%%%%%%%%%%%%%%%%%%%%%%%

\end{document}